\documentclass{amsart}
\usepackage{hyperref}
\usepackage{txfonts}
\usepackage{amsfonts}

\newtheorem{theorem}{Theorem}[section]

\newtheorem{corollary}[theorem]{Corollary}

\theoremstyle{definition}

\theoremstyle{remark}
\newtheorem{remark}[theorem]{Remark}

\numberwithin{equation}{section}

\begin{document}

\title{A relative Toponogov comparison theorem}

\author{}
\address{}
\curraddr{}
\email{}
\thanks{}

\author{Jianming Wan}
\address{School of Mathematics, Northwest University, Xi'an 710127, China}
\email{wanj\_m@aliyun.com}
\thanks{The author is supported by Natural Science Foundation of Shaanxi Province of China (No.2022JM-010) and Shaanxi Fundamental Science Research Project for Mathematics and Physics (No.22JSY025)}


\date{April 12, 2023}

\dedicatory{}


\begin{abstract}
We present a relative form of the Toponogov comparison theorem.
\end{abstract}
\maketitle




\section{Main result}
Let $M$ be a complete Riemannian manifold with sectional curvature $sec_{M}\geq k, (k=0, -1, 1)$ and $S^{n}_{k}$ the simply connected spaces form. Let $\gamma_{1}, \gamma_{2}$ be two geodesic segments in $M$ such that $\gamma_{1}(0)=\gamma_{2}(0)$ and $\bar{\gamma}_{1}, \bar{\gamma}_{2}$ be two geodesic segments in $S^{n}_{k}$ such that $\bar{\gamma}_{1}(0)=\bar{\gamma}_{2}(0)$. $\measuredangle(\gamma_{1}^{'}(0), \gamma_{2}^{'}(0))=\measuredangle(\bar{\gamma}_{1}^{'}(0), \bar{\gamma}_{2}^{'}(0))=\alpha$. Let $r(t)=d_{M}(\gamma_{1}(b), \gamma_{2}(t))$ and $\bar{r}(t)=d_{S^{n}_{k}}(\bar{\gamma}_{1}(b), \bar{\gamma}_{2}(t))$.  Toponogov comparison theorem \cite{[T]} claims that $$r(t)\leq\bar{r}(t).$$
It is an important tool in Riemannian geometry. One can find its proof in some standard textbooks of Riemannian geometry, such as \cite{[CE]} and \cite{[P]}.

We shall prove the following relative form of Toponogov comparison theorem.
\begin{theorem}
(A) The distance ratio $$t\mapsto\frac{r(t)}{\bar{r}(t)}$$ is a non-increasing function for i): $t\leq b$ when $k=0$; ii): $t\leq b\leq \pi/2$  when $k=1$.

(B) The distance difference $$\psi(t)=\bar{r}(t)-r(t)$$ is a non-decreasing function for i): $t\leq b$ when $k=0,-1$; ii): $t\leq b\leq \pi/2$  when $k=1$.
\end{theorem}

The word ``relative" comes from Bishop-Gromov's relative volume comparison. The proof of Theorem 1.1 is based on frequent use of Toponogov comparison theorem.

One can see that in case $k=0, 1$, conclusion (A) implies (B). For $t_{1}<t_{2}$, the non-increasing of $\frac{r(t)}{\bar{r}(t)}$ and $r(t)\leq\bar{r}(t)$ lead to $$\frac{r(t_{2})-r(t_{1})}{\bar{r}(t_{2})-\bar{r}(t_{1})}\leq \frac{r(t_{2})}{\bar{r}(t_{2})}\leq 1.$$
It follows that $\bar{r}(t_{2})-r(t_{2})\geq\bar{r}(t_{1})-r(t_{1})$.

In conclusion (A), the present method can not work for the case $k=-1$.

\section{a proof of main result}
Let $\gamma_{3}$ be a geodesic segment from $\gamma_{1}(b)$ to $\gamma_{2}(t)$ and $\bar{\gamma}_{3}$ be the geodesic segment from $\bar{\gamma}_{1}(b)$ to $\bar{\gamma}_{2}(t)$. Write $$\beta=\measuredangle(\gamma_{2}^{'}(t), \gamma_{3}^{'}(r(t)), \ \bar{\beta}=\measuredangle(\bar{\gamma}_{2}^{'}(t), \bar{\gamma}_{3}^{'}(r(t)).$$
If $\gamma_{2}(t)$ is not a cut point of $\gamma_{1}(b)$, then there exists $\delta>0$ such that $r(t)$ is smooth on $(t-\delta,t+\delta)$. The first variation formula yields $$r^{'}(t)=\cos \beta, \ \bar{r}^{'}(t)=\cos \bar{\beta}.$$
In case $k=0, 1$, we would show $$(\frac{r(t)}{\bar{r}(t)})^{'}=\frac{r^{'}\bar{r}-r\bar{r}^{'}}{\bar{r}^{2}}=\frac{1}{\bar{r}^{2}}(\bar{r}\cos\beta-r\cos\bar{\beta})\leq0.$$
Equivalently, $$r\cos\bar{\beta}-\bar{r}\cos\beta\geq0.$$ And in case $k=-1$, we would show $$\psi^{'}(t)=\bar{r}^{'}(t)-r^{'}(t)=\cos\bar{\beta}-\cos\beta\geq0.$$

\textbf{(1) Case $k=0$.} The law of Cosines gives  $$\cos\bar{\beta}=\frac{\bar{r}^{2}+t^{2}-b^{2}}{2\bar{r}t}.$$
From Toponogov comparison theorem, we have $b^{2}\leq r^{2}+t^{2}-2rt\cos\beta$. This implies $$\cos\beta\leq\frac{r^{2}+t^{2}-b^{2}}{2rt}.$$
Then
$$r\cos\bar{\beta}-\bar{r}\cos\beta\geq\frac{\bar{r}^{2}-r^{2}}{2r\bar{r}t}(b^{2}-t^{2})\geq0,$$
when $t\leq b$.

\textbf{(2) Case $k=1$.} The law of Cosines gives  $$\cos\bar{\beta}=\frac{\cos b-\cos\bar{r}\cos t}{\sin\bar{r} \sin t}.$$
From Toponogov comparison theorem, we have $\cos b\geq \cos r\cos t+\sin r \sin t\cos\beta$. This implies $$\cos\beta\leq\frac{\cos b-\cos r\cos t}{\sin r \sin t}.$$
Then
\begin{eqnarray*}
r\cos\bar{\beta}-\bar{r}\cos\beta &\geq&\frac{1}{\sin\bar{r} \sin r\sin t}[(\bar{r}\sin\bar{r}\cos r-r\sin r\cos\bar{r})\cos t+(r\sin r-\bar{r}\sin\bar{r})\cos b]\\
&\geq& \frac{\cos b}{\sin\bar{r} \sin r\sin t}(\bar{r}\sin\bar{r}\cos r-r\sin r\cos\bar{r}+r\sin r-\bar{r}\sin\bar{r})\\
&=& \frac{r\bar{r}\cos b}{\sin t}(\frac{1-\cos \bar{r}}{\bar{r}\sin \bar{r}}-\frac{1-\cos r}{r\sin r})\\
&\geq& 0.
\end{eqnarray*}
when $t\leq b\leq\pi/2$.

The second $``\geq"$ holds because the function $\phi(\bar{r})=\bar{r}\sin\bar{r}\cos r-r\sin r\cos\bar{r}\geq0$ for $0<r\leq\bar{r}\leq\pi$.  To see this, we write $$\phi(\bar{r})=r\bar{r}\sin r\sin\bar{r}(\frac{\cos r}{r\sin r}-\frac{\cos \bar{r}}{\bar{r}\sin \bar{r}}).$$ Since the function $\frac{\cos t}{t\sin t}$ is decreasing for $0<t\leq\pi$, we have $\phi(\bar{r})\geq0$.

The third $``\geq"$ holds because $f(t)=\frac{1-\cos t}{t\sin t}$ is increasing ($f^{'}(t)>0$) for $0<t<\pi$.

\textbf{(3) Case $k=-1$.} The law of Cosines gives  $$\cos\bar{\beta}=\frac{\cosh\bar{r}\cosh t-\cosh b}{\sinh\bar{r} \sinh t}.$$
From Toponogov comparison theorem, we have $\cosh b\leq \cosh r\cosh t-\sinh r \sinh t\cos\beta$. This implies $$\cos\beta\leq\frac{\cosh r\cosh t-\cosh b}{\sinh r \sinh t}.$$
Then
\begin{eqnarray*}
\cos\bar{\beta}-\cos\beta&\geq&\frac{1}{\sinh\bar{r} \sinh r\sinh t}[\sinh(r-\bar{r})\cosh t+(\sinh\bar{r}-\sinh r)\cosh b]\\
&\geq& \frac{\cosh b}{\sinh\bar{r} \sinh r\sinh t}[\sinh(r-\bar{r})+\sinh\bar{r}-\sinh r],
\end{eqnarray*}
when $t\leq b$. One is easy to see that the function $$f(\bar{r})=\sinh(r-\bar{r})+\sinh\bar{r}-\sinh r$$ satisfies $f(r)=0, \frac{df}{d\bar{r}}\geq0$ for $\bar{r}\geq r$. So $f(\bar{r})\geq 0$. Hence $$\cos\bar{\beta}-\cos\beta\geq0.$$

Unfortunately, we can not obtain the conclusion (A) when $k=-1$. In this situation,
$$r\cos\bar{\beta}-\bar{r}\cos\beta \geq \frac{r\bar{r}\cosh b}{\sinh t}(\frac{\cosh \bar{r}-1}{\bar{r}\sinh \bar{r}}-\frac{\cosh r-1}{r\sinh r}).$$
Since $\frac{\cosh t-1}{t\sinh t}$ is decreasing, the right hand is less than 0.

 If $\gamma_{2}(t)$ is a cut point of $\gamma_{1}(b)$, there would be more than one geodesic segment from $\gamma_{1}(b)$ to $\gamma_{2}(t)$.  From Petersen \cite{[P]} (Page 224, Exercise 5.9.28.), the right-hand derivative  $$r^{'}_{+}(t)=\min \cos\beta$$  and left-hand derivative $$r^{'}_{-}(t)=\max \cos\beta.$$
By the calculation of above three cases, we have $$(\frac{r(t)}{\bar{r}(t)})^{'}_{+}\leq0,\ (\frac{r(t)}{\bar{r}(t)})^{'}_{-}\leq0$$ and $$\psi^{'}_{+}(t)\geq0,\ \psi^{'}_{-}(t)\geq0.$$

To sum up, whether or not $r(t)$ is smooth, we always have $(\frac{r(t)}{\bar{r}(t)})^{'}_{+}\leq0,\ (\frac{r(t)}{\bar{r}(t)})^{'}_{-}\leq0$ and $\psi^{'}_{+}(t)\geq0,\ \psi^{'}_{-}(t)\geq0.$
Then we can complete the proof of Theorem 1.1 from the fact (see Miller-Vyborny \cite{[MV]}): Let $f$ be a continuous function on $[a,b]$. If for each $x\in (a,b)$ one of the one-sided derivative $f^{'}_{+}$ or $f^{'}_{-}$ exists, and is nonnegative (possibly $+\infty$), then $f$ is monotonic increasing.

\begin{remark}
One may think that the restriction $t\leq b$ in Theorem 1.1 is not necessary. But the proof shows that non-decreasing of $\psi(t)$ is equivalent to $\beta\geq\bar{\beta}$. It seems that no reason makes this true globally.
\end{remark}

\begin{remark}
If $t>b$, we can compare along $\gamma_{1}$. We denote $r(t,s)=d_{M}(\gamma_{1}(t), \gamma_{2}(s))$ and $\bar{r}(t,s)=d_{S^{n}_{k}}(\bar{\gamma}_{1}(t), \bar{\gamma}_{2}(s))$. Then conclusion (A) in Theorem 1.1 can be written as $$\frac{r(t, s_{1})}{\bar{r}(t, s_{1})}\geq\frac{r(t, s_{2})}{\bar{r}(t, s_{2})}$$ when $s_{1}< s_{2}\leq t$ and
$$\frac{r(t_{1}, s)}{\bar{r}(t_{1}, s)}\geq\frac{r(t_{2}, s)}{\bar{r}(t_{2}, s)}$$ when $t_{1}< t_{2}\leq s$.

Denote $\psi(t,s)=\bar{r}(t,s)-r(t,s)$. Conclusion (B) says
 $$\psi(t,s_{1})\leq\psi(t,s_{2})$$ when $s_{1}< s_{2}\leq t$ and  $$\psi(t_{1},s)\leq\psi(t_{2},s)$$ when $t_{1}< t_{2}\leq s$.

So Theorem 1.1 can be flexible for some possible applications.
\end{remark}

\section{start point free case}
In this section we consider the relative Toponogov comparison theorem when the start point is free. Now we set $r^{*}(t)=d_{M}(\gamma_{1}(t), \gamma_{2}(t))$ and $\bar{r}^{*}(t)=d_{S^{n}_{k}}(\bar{\gamma}_{1}(t), \bar{\gamma}_{2}(t))$. Here $\gamma_{1}(t), \gamma_{2}(t), \bar{\gamma}_{1}(t), \bar{\gamma}_{2}(t)$ are same to that in Section 1.

Then we have
\begin{theorem}
(A) The distance ratio $$t\mapsto\frac{r^{*}(t)}{\bar{r}^{*}(t)}$$ is a non-increasing function for i): $t\geq0$ when $k=0$; ii): $t\leq\pi/2$  when $k=1$.

(B) The distance difference $$\psi^{*}(t)=\bar{r}^{*}(t)-r^{*}(t)$$ is  a non-decreasing function for i): $t\geq0$ when $k=0,-1$; ii): $t\leq \pi/2$  when $k=1$.
\end{theorem}

The proof is similar to that of Theorem 1.1.

In addition, we write $\gamma=\measuredangle(\gamma_{1}^{'}(t), -\gamma_{3}^{'}(0)$ and
$\bar{\gamma}=\measuredangle(\bar{\gamma}_{1}^{'}(t), -\bar{\gamma}_{3}^{'}(0))=\bar{\beta}$.

If $\gamma_{2}(t)$ is not a cut point of $\gamma_{1}(t)$, then $r^{*}(t)$ is smooth, the first variation formula yields $$r^{*'}(t)=\cos\beta+\cos\gamma,\ \bar{r}^{*'}(t)=2\cos \bar{\beta}.$$
We can show $$(\frac{r^{*}(t)}{\bar{r}^{*}(t)})^{'}=\frac{r^{*'}\bar{r}^{*}-r^{*}\bar{r}^{*'}}{\bar{r}^{*2}}=\frac{\bar{r}^{*}(\cos\beta+\cos\gamma)-2r^{*}\cos\bar{\beta}}{\bar{r}^{*2}}\leq0.$$
when $k=0, 1$ and $$\psi^{*'}(t)=\bar{r}^{*'}(t)-r^{*'}(t)=2\cos\bar{\beta}-(\cos\beta+\cos\gamma)\geq0.$$
when $k=-1$.

\textbf{(1) Case $k=0$.} Note that $\cos\bar{\beta}=\frac{\bar{r}^{*}}{2t},\ \cos\beta\leq\frac{r^{*}}{2t},\ \cos\gamma\leq\frac{r*}{2t}.$
Then $$2r^{*}\cos\bar{\beta}-\bar{r}^{*}(\cos\beta+\cos\gamma)\geq0.$$

\textbf{(2) Case $k=1$.} Note that $$\cos\bar{\beta}=\frac{\cos t(1-\cos\bar{r}^{*})}{\sin\bar{r}^{*} \sin t}$$ and $$\cos\beta,\ \cos\gamma\leq\frac{\cos t(1-\cos r^{*})}{\sin r^{*} \sin t}.$$
Then
\begin{eqnarray*}
2r^{*}\cos\bar{\beta}-\bar{r}^{*}(\cos\beta+\cos\gamma) &\geq& \frac{2r^{*}\bar{r}^{*}\cos t}{\sin t}(\frac{1-\cos \bar{r}^{*}}{\bar{r}^{*}\sin \bar{r}^{*}}-\frac{1-\cos r^{*}}{r^{*}\sin r^{*}})\\\\
&\geq&0.
\end{eqnarray*}
when $t\leq\pi/2$.

\textbf{(3) Case $k=-1$.} Note that $$\cos\bar{\beta}=\frac{(\cosh\bar{r}^{*}-1)\cosh t}{\sinh\bar{r}^{*} \sinh t}$$ and $$\cos\beta, \ \cos\gamma\leq\frac{(\cosh r^{*}-1)\cosh t}{\sinh r^{*} \sinh t}.$$
Then
\begin{eqnarray*}
2\cos\bar{\beta}-(\cos\beta+\cos \gamma) &\geq& \frac{2\cosh t}{\sinh\bar{r}^{*} \sinh r^{*}\sinh t}[\sinh(r^{*}-\bar{r}^{*})+\sinh\bar{r}^{*}-\sinh r^{*}]\\
&\geq&0.
\end{eqnarray*}

If $\gamma_{2}(t)$ is a cut point of $\gamma_{1}(t)$, there would be more than one geodesic segment from $\gamma_{1}(t)$ to $\gamma_{2}(t)$. Using similar arguments in  Petersen \cite{[P]} (Page 224, Exercise 5.9.28.), we can show $$r^{*'}_{+}(t)=\min (\cos\beta+\cos\gamma),\ r^{*'}_{-}(t)=\max (\cos\beta+\cos\gamma).$$   So $$(\frac{r^{*}(t)}{\bar{r}^{*}(t)})^{'}_{+}\leq 0, \ (\frac{r^{*}(t)}{\bar{r}^{*}(t)})^{'}_{-}\leq 0$$ and
$$\psi^{*'}_{+}(t)\geq0, \ \psi^{*'}_{-}(t)\geq0.$$

By the result of Miller-Vyborny \cite{[MV]}, we complete the proof of Theorem 3.1.

When $M$ has nonnegative sectional curvature. From Theorem 1.1 and 3.1, we obtain
\begin{corollary}
(1) Let $d=d_{M}(\gamma_{1}(b), \gamma_{2}(b))$ and $\measuredangle(\gamma_{1}^{'}(0), \gamma_{2}^{'}(0))=\alpha$. Then $$r(t)\geq d\cos\frac{\alpha}{2}.$$

(2) Let $d_{i}=d_{M}(\gamma_{1}(l_{i}), \gamma_{2}(l_{i})), i=1,2, l_{1}<l_{2}$. Then $$d_{1}\geq\frac{l_{1}}{l_{2}}d_{2}.$$
\end{corollary}

\begin{proof}
(1): By (A) of Theorem 1.1, $r\geq\frac{\bar{r}}{\bar{d}}d\geq d\sin\bar{\beta}=d\cos\frac{\alpha}{2}$.

(2): By (A) of Theorem 3.1, $d_{1}\geq\frac{\bar{d}_{1}}{\bar{d}_{2}}d_{2}=\frac{l_{1}}{l_{2}}d_{2}$.
\end{proof}

\bibliographystyle{amsplain}

\begin{thebibliography}{10}

\bibitem {[CE]} Cheeger, J.; Ebin, D. G. \textit{Comparison theorems in Riemannian geometry.} North-Holland Mathematical Library, Vol. 9. North-Holland Publishing Co., Amsterdam-Oxford; American Elsevier Publishing Co., Inc., New York, 1975. viii+174 pp.

\bibitem {[MV]} Miller, A. D.; Vyborny, R. \textit{Some Remarks on Functions with One-Sided Derivatives.}
Amer. Math. Monthly 93 (1986), no. 6, 471-475.


\bibitem {[P]} Petersen, P. \textit{Riemannian geometry, Third Edition.}  Graduate Texts in Mathematics, 171. Springer, Cham, 2016. xviii+499 pp.

\bibitem {[T]} Toponogov, V. A. \textit{Riemann spaces with curvature bounded below. (Russian)} Uspehi Mat. Nauk 14 1959 no. 1 (85), 87-130.




\end{thebibliography}

\end{document}